\documentclass[11pt]{article}
\usepackage[english]{babel}
\usepackage[T1]{fontenc}
\usepackage[latin1]{inputenc}

\usepackage{amsfonts,amssymb,amsmath,amsthm, graphics, setspace}
\usepackage{bbm}
\usepackage{xcolor}
\usepackage{graphics}
\usepackage{graphicx}
\usepackage{mathtools}
\usepackage{hyperref} 
\usepackage{soul}

\newtheorem{theorem}{Theorem}[section]

\newtheorem{proposition}[theorem]{Proposition}

\newtheorem*{theoremnonum}{Theorem}

\theoremstyle{remark}

\numberwithin{equation}{section}

\title{Genetic contribution of an advantaged mutant in the biparental Moran model - finite selection}
\author{Camille Coron, Yves Le Jan}
\date{}

\usepackage{color}
\usepackage{amsmath}
\usepackage{amsthm}

\begin{document}
\maketitle

\begin{abstract}
We consider a population of $N$ individuals, whose dynamics through time is represented by a biparental Moran model with two types: an advantaged type and a disadvantaged type. The advantage is due to a mutation, transmitted in a Mendelian way from parent to child that reduces the death probability of individuals carrying it. We assume that initially this mutation is carried by a proportion $a$ of individuals in the population. Once the mutation is fixed, a gene is sampled uniformly in the population, at a locus independent of the locus under selection. We then give the probability that this gene initially comes from an advantaged individual, i.e. the genetic contribution of these individuals, as a function of $a$ and when the population size is large.
\end{abstract}

\textbf{Keywords:} Biparental Moran model with selection ; Ancestor's genetic contribution.

\section{Motivation and model}

\subsection{Motivation}
In this article we aim at understanding how the advantage conferred to an individual by a genetic mutation influences its contribution to the genome of a population after a long time. This article follows two other works (\cite{geneal} and \cite{CoronLeJan22}) that focus on the contribution of ancestors to the genome of a sexually reproducing population, in which each individual has two parents who contribute equally to the genome of their offspring.

More precisely, we consider a population of haploid individuals reproducing sexually, i.e. for which the genome of each individual is a random mixture of the genome of its two parents. We assume that initially a proportion $a$ of these individuals carry a mutation at one locus, and that individuals carrying this mutation have an increased life expectancy, and therefore are advantaged, regarding genome transmission. Our aim is to study the long time effect of this mutation on the genetic composition of the population, when population size is large, and in particular to quantify the impact of the strength of selection, on the genetic contribution of an ancestor, to the population.

Biparental genealogies have received some interest, notably in \cite{Chang1999,Derrida2000,GravelSteel2015}, in which time to recent common ancestors and ancestors' weights are investigated for the Wright-Fisher biparental model. In \cite{geneal}, we studied the asymptotic law of the  contribution of an ancestor, to the genome of the present time population. The articles \cite{MatsenEvans2008} and \cite{BartonEtheridge2011} study the link between pedigree, individual reproductive success and genetic contribution. The monoparental Moran model with selection at birth has received some interest, notably in \cite{EtheridgeGriffiths} that studies its dual coalescent and \cite{KluthBaake} that notably studies alleles fixation probabilities and ancestral lines. Finally, the limiting case where the strength of selection is infinite was studied in \cite{CoronLeJan22}.

\subsection{Model}

As in the previous papers \cite{geneal} and \cite{CoronLeJan22}, we consider a population of $N$ individuals whose dynamics through time is modeled using a Moran biparental model. However this model is modified in order to take into account the advantage conferred by a mutation of interest. We will in fact assume that selection impacts only the death of individuals. 

More precisely, we assume that the population is composed of a fixed number $N$ of individuals. Next we assume that at a given locus, a mutation confers some advantage to the individuals that carry it. Advantaged individuals (that carry the mutation) are characterized by a death weight $1$ whereas non advantaged individuals are characterized by a death weight $1+s$. At each discrete time step, two individuals are chosen independently and uniformly to be parents, and produce one offspring. This offspring replaces a third individual, chosen with a probability proportional to its death weight. Therefore at each time step, a given non-advantaged individual has a probability to die that is $1+s$ times more important than that of a given advantaged individual. Hence advantaged individuals have an increased life expectancy and consequently a larger mean offspring number. 

Genetic transmission is assumed to follow Mendel rules, which means that for a given locus, one of the two parents is chosen uniformly (among the two) and transmits its allele (i.e. a copy of its genome at this locus) to the offspring. This transmission is not independent for loci that are close on the genome, but can be considered as independent for loci that are on different chromosomes for instance. In particular the transmission of advantage to offspring is assumed to be characterized by Mendelian transmission at one locus. In other words, the offspring inherits the level of advantage of one of its parents, chosen uniformly at random among its two parents. By convention, we call "mother" the parent that transmits its allele at the locus under mutation, and "father" the other parent. Note that the limiting case where $s$ is equal to infinity (therefore the individual that dies is always chosen among non advantaged individuals) has been studied in the previous paper \cite{CoronLeJan22}. Comparisons to this previous work will be provided later on.

Let us denote by $I=\{1,2,...,N\}$ the sites in which individuals live, and denote by $(\mu_n,\pi_n,\kappa_n)\in I^3$ the respective positions of the mother, father, and offspring at time step $n$.
As in \cite{geneal}, this reproduction dynamics defines an oriented random graph on $I\times\mathbb{Z}_+$ (as represented in Figure \ref{FigPedigree}), denoted $\mathcal{G}_N$, representing the pedigree of the population, such that between time $n$ and time $n+1$, two arrows are drawn from $(\kappa_n,n+1)$ to $(\pi_n,n)$ and $(\mu_n,n)$ respectively and $N-1$ arrows are drawn from $(i,n+1)$ to $(i,n)$ for each $i\in I\setminus\{\kappa_n\}$. Note that individuals are now characterized by their advantage : advantaged individuals are represented in red. We finally denote by $\mathcal{Y}_n\subset \{1,...,N\}$ the set of advantaged individuals at time $n$ and denote by $\{\mathcal{F}_n,n\in\mathbb{Z}_+\}$ the natural filtration associated to the stochastic process $(\mu_n,\pi_n,\kappa_n,\mathcal{Y}_n)_{n\in\mathbb{Z}_+}$.

\begin{figure}[ht]
\begin{center}\includegraphics[scale=0.5]{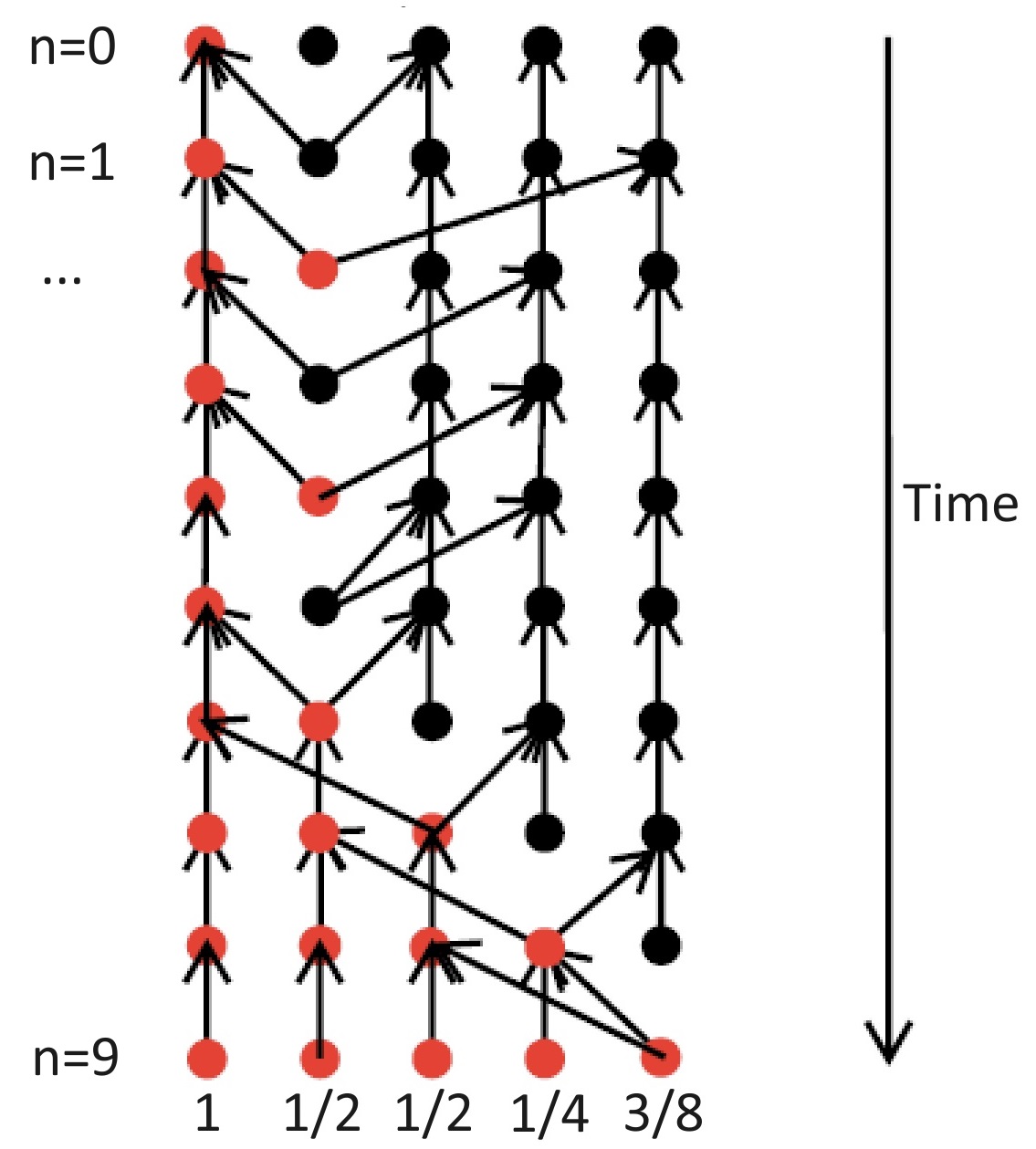}
\end{center} 
\caption{Graph $\mathcal{G}_N$ representing the pedigree of a population with $5$ individuals, during $10$ time steps. The time orientation is from past to future and arrows materialize gene flow between individuals, when going backward in time. Advantaged individuals are represented in red. Numbers at the bottom gives the probability that a gene sampled in each of the individuals come from the initially advantaged individual. In this example the genetic weight of the initially advantaged individual (i.e. the probability that a gene sampled uniformly at time $n=9$ comes from this individual at time $n=0$) is equal to $21/40=1/5(1+1/2+1/2+1/4+3/8).$\label{FigPedigree}}
\end{figure}

\subsection{Ancestors genetic weights}
 To study the impact of selection on the genetic composition of the population we now consider a new locus, distant enough from the locus under mutation, so that the genome is assumed to be transmitted independently at these two loci. Then we consider a gene sampled in one of the individuals present in the population at time $n$ and our interest is to study the probability for this gene, to come from each of the individuals living at time $0$. The genealogy of this gene (i.e. the individual in which a copy of this gene was present, assuming no mutation and no recombination, at each time $t=n-k\leq n$), denoted by $(X^{(n)}_k, n-k)_{0\leq k\leq n}$, is a random walk on the pedigree $\mathcal{G}_N$, starting from the position $(i,n)$. 

The key element in this model, as introduced in \cite{geneal} is therefore the random variable  

\begin{equation} \label{eq:defA} A_n(i,j)=\mathbb{P}(X^{(n)}_n=j|X^{(n)}_0=i,\mathcal{G}_N).\end{equation}

This quantity is a random variable, since it is a deterministic function of the random graph $\mathcal{G}_N$. This quantity, given the genealogy $\mathcal{G}_N$, is the probability that any gene of individual $i$ living in generation $n$ comes from ancestor $j$ living at generation $0$. Illustrations are given in Figure \ref{FigPedigree}. If genome size is very large and the evolutions of distant genes are sufficiently decorrelated, we can expect this quantity to be close to the proportion of genes of individual $i$ that come from individual $j$. This quantity will also be called the weight of the ancestor $j$  in the genome of individual $i$.

\subsection{Main result}\label{sec:IntroResult}
Let us denote by $(\mathcal{Y}_n)_{n\in\mathbb{Z}_+}$ the set of advantaged individuals at time step $n$, and $(Y_n=|\mathcal{Y}_n|)_{n\in\mathbb{Z}_+}$ the number of advantaged individuals at each time $n\in\mathbb{Z}_+$. 

We are interested in the impact of selection on the weight of ancestors, therefore on the probability that a gene sampled in the population comes from an advantaged individual. 
This comes back to studying the two following quantities : 
$$
\Xi^A_n=\frac{\sum_{l\in\mathcal{Y}_n}\sum_{l'\in\mathcal{Y}_0} A_n(l,l')}{Y_n}
$$
which is a random variable (as deterministic function of the pedigree between time $0$ and time $n$) giving the probability that a gene sampled uniformly among advantaged individuals at time $n$ comes from an initially advantaged individual, given the pedigree between time $0$ and time $n$, and
$$
\Xi^B_n=\frac{\sum_{l\notin\mathcal{Y}_n}\sum_{l'\in\mathcal{Y}_0} A_n(l,l')}{N-Y_n}
$$
gives the probability that a gene sampled uniformly among non-advantaged individuals at time $n$ comes from an initially advantaged individual, given the pedigree between time $0$ and time $n$. These two random quantities can be then seen as the weight of advantaged individuals among advantaged and disadvantaged individuals, respectively. An example of weight of the initially advantaged individuals is given in Figure \ref{FigPedigree}.

We will focus on the expectation of these two random variables, and more specifically on the expectation of the first one once the advantageous mutation is fixed (i.e. once $Y_n=N$ which happens at a stopping time we denote by $T^Y_N$), when we start from a fixed proportion $a$ of individual carrying this mutation. 

Our main result is the following : 
\begin{theoremnonum}[Theorem \ref{thm}]
\begin{align*}\mathbb{E}(\Xi^A_{T^Y_N}\mathbf{1}_{T^Y_N<\infty}|Y_0=\lfloor aN\rfloor)&\rightarrow_{N\rightarrow\infty} \frac{a^{\frac{1+s}{2s}}}{(1-a)^{\frac{1}{2s}}}\left(\int_{a}^{1}\frac{(1-u)^{\frac{1}{2s}}}{u^{\frac{1+s}{2s}}}\left[\frac{1}{2}+\frac{1}{2s}\frac{1}{1-u}\right]du\right) \end{align*}
\end{theoremnonum}
It can be shown that this quantity increases with $s$ and converges to $2\sqrt{a}
-a$ as $s\uparrow \infty$. This result (for $s=\infty$) can also be retrieved from \cite{CoronLeJan22}.
As an example, for large enough $s$, an initial proportion of advantaged individuals equal to $a=1\%$ would yield an average asymptotic genetic weight close to $19\%$ for this set of $1\%$ initially advantaged individuals.

\section{Results and proofs}
\subsection{Number of advantaged individuals dynamics}

%

Recall that $(\mathcal{Y}_n)_{n\in\mathbb{Z}_+}$ is the set of advantaged individuals at time step $n$, and $(Y_n)_{n\in\mathbb{Z}_+}=|\mathcal{Y}_n|$ the number of advantaged individuals at each time step $n\in\mathbb{Z}_+$. This number of advantaged individuals satisfies the following

\begin{proposition}\label{prop:Y}
The stochastic process  $(Y_n)_{n\in\mathbb{Z}_+}$ is a Markov chain such that if $Y_n=k\in \{0,1,...,N\}$ then $Y_{n+1}\in\{k-1,k,k+1\}$, and $\mathbb{P}(Y_{n+1}=k-1|Y_n=k)=p_k\times\frac{1}{2+s}$, $\mathbb{P}(Y_{n+1}=k+1|Y_n=k)=p_k\times\frac{1+s}{2+s}$ and $\mathbb{P}(Y_{n+1}=k|Y_n=k)=1-p_k$, where $$p_k=\frac{k(N-k)}{N\left(\frac{1}{2+s}k+\frac{1+s}{2+s}(N-k)\right)}.$$ This Markov chain is absorbed in $0$ and in $N$.
\end{proposition}

\begin{proof} At each time step, as only one individual dies and one individual arises, the number of advantaged individual can only be increased or decreased by $1$, or remain the same. Now if the number of advantaged individual is equal to $k$, then the probability for the individual that dies at this time step to be advantaged, is equal to $\frac{k}{k+(1+s)(N-k)}$. Now for the Markov chain $(Y_n)_{n\in\mathbb{Z}_+}$ to decrease, one needs that the mother is a non-advantaged individual, while the replaced individual is advantaged. The first event occurs with probability $(N-k)/N$ while the second event occurs with probability $\frac{k}{k+(1+s)(N-k)} $, therefore $\mathbb{P}(Y_{n+1}=k-1|Y_n=k)=(N-k)/N\times\frac{k}{k+(1+s)(N-k)}=\frac{1}{2+s}\times \frac{k(N-k)}{N(\frac{1}{2+s}k+\frac{1+s}{2+s}(N-k))}$. Similarly, for the Markov chain $(Y_n)_{n\in\mathbb{Z}_+}$ to increase, one needs that the mother is an advantaged individual, while the replaced individual is non-advantaged. The first event occurs with probability $k/N$ while the second event occurs with probability $\frac{(1+s)(N-k)}{k+(1+s)(N-k)} $, therefore $\mathbb{P}(Y_{n+1}=k-1|Y_n=k)=k/N\times\frac{(1+s)(N-k)}{k+(1+s)(N-k)}=\frac{1+s}{2+s}\times \frac{k(N-k)}{N(\frac{1}{2+s}k+\frac{1+s}{2+s}(N-k))}$. Finally for the Markov chain $Y_n$ to stay at position $k$, one needs that both the mother and replaced individual are either both advantaged, or both disadvantaged, which gives that $\mathbb{P}(Y_{n+1}=k|Y_n=k)=\frac{k}{N}\frac{k}{k+(1+s)(N-k)}+\frac{N-k}{N}\frac{(1+s)(N-k)}{k+(1+s)(N-k)}=\frac{k^2+(1+s)(N-k)^2}{N(k+(1+s)(N-k))}=1-\frac{(2+s)k(N-k)}{N(k+(1+s)(N-k))}=1-p_k$. As $p_0=p_N=0$, this gives that the states $0$ and $N$ are absorbing.
\end{proof}

\medskip

\textbf{Notations : }
For any $y\in\{0,1,...,N\}$, we define  $T^Y_{y}=\inf\{n:Y_n=y\}$.
Set also $\tau_0=0$ and for all $n\in\mathbb{Z}_+^*$,
$\tau_n=\inf\{l>\tau_{n-1}:Y_l\neq Y_{\tau_{n-1}}\}$ which is the $n$-th jump time of the Markov chain $Y$. Then let us set for any $n\leq \sup\{n:\tau_n<\infty\}$, $$Z_n=Y_{\tau_n},$$ which is the sequence of states visited by the Markov chain $(Y_n)_{n\in\mathbb{Z}_+}$, and set for any $y\in\{0,1,...,\}$, $T^Z_y=\tau_{T_y^Y}=\inf\{n:Z_n=y\}$, and $T^Z_{y,z}=\inf\{n:Z_n\in\{y,z\}\}$.

Then the stochastic process $(Z_n)_{n\in\mathbb{N}}$ follows the 

\begin{proposition}\label{prop-Z}
The stochastic process $(Z_n)_{n\in\mathbb{N}}$ is a simple biased random walk absorbed in $0$ and $N$ : as long as $Z_n\in\{1,...,N-1\}$, $Z_{n+1}\in\{Z_n-1,Z_n+1\}$ and $$\mathbb{P}(Z_{n+1}=Z_n+1)=\frac{1+s}{2+s}.$$
\end{proposition}

\begin{proof} By decomposing according to the duration spent by the Markov chain $(Y_n)_{n\in\mathbb{N}}$ in state $k\in[\![1,N]\!]$, one has :
$$\mathbb{P}(Z_{n+1}=Z_n+1|Z_n=k)=\sum_{l=0}^{\infty}(1-p_k)^l\times \frac{1+s}{2+s}\times p_k=\frac{1+s}{2+s}.$$
\end{proof}

\medskip

\subsection{Mean weight of advantaged individuals}


Our aim in this work is to quantify the effect of selection on the contribution of an ancestor to the genome of the population. As mentioned in Section \ref{sec:IntroResult}, we therefore focus on the following quantities :

\begin{equation*}
\Xi^A_n=\frac{\sum_{l\in\mathcal{Y}_n}\sum_{l'\in\mathcal{Y}_0} A_n(l,l')}{Y_n}\quad\quad \text{and} \quad\quad
\Xi^B_n=\frac{\sum_{l\notin\mathcal{Y}_n} \sum_{l'\in\mathcal{Y}_0}A_n(l,l')}{N-Y_n}
\end{equation*}

that are respectively the probability that a gene sampled uniformly among advantaged (resp. disadvantaged) individuals, comes from an initially advantaged individual. 
Indeed, denoting by  $\mathcal{U}(\mathcal{S})$ the uniform law on any set $\mathcal{S}\subseteq I$,

\begin{equation}
\begin{aligned}
\sum_{l\in\mathcal{Y}_n}\sum_{l'\in\mathcal{Y}_0} A_n(l,l')&=\sum_{l\in\mathcal{Y}_n}\sum_{l'\in\mathcal{Y}_0} \mathbb{P}(X_n^{(n)}=l'|X_0^{(n)}=l,\mathcal{G}_N)\\&=\sum_{l\in\mathcal{Y}_n}\mathbb{P}(X_n^{(n)}\in\mathcal{Y}_0|X_0^{(n)}=l,\mathcal{G}_N)\\&=Y_n\mathbb{P}(X_n^{(n)}\in\mathcal{Y}_0|X_0^{(n)}\sim\mathcal{U}(\mathcal{Y}_n),\mathcal{Y}_n,\mathcal{G}_N).
\end{aligned}
\end{equation}
 Therefore $$\Xi^A_n=\mathbb{P}(X_n^{(n)}\in\mathcal{Y}_0|X_0^{(n)}\sim\mathcal{U}(\mathcal{Y}_n),\mathcal{Y}_n,\mathcal{G}_N),$$ and similarly
$$
\Xi^B_n=\mathbb{P}(X_n^{(n)}\in\mathcal{Y}_0|X_0^{(n)}\sim\mathcal{U}(\mathcal{Y}^c_n),\mathcal{Y}_n,\mathcal{G}_N).
$$

These two quantities give a measure of the quantity of genome that comes from the initially advantaged population, respectively among advantaged and disadvantaged individuals, knowing the pedigree, i.e. the parental relationship between individuals. We are particularly interested in the order of magnitude of $\Xi^A_n$ and $\Xi^B_n$, when $n$ goes to infinity, and when the population is large. The case of greatest interest is when the advantageous mutation ends up invading all the population, i.e. when $T^Y_N<\infty$, which happens with probability almost $1$ as soon as $Y_0\geq a N$ for a positive $a$ and large $N$. In this article we will then aim at studying $\mathbb{E}(\Xi^A_{T^{Y}_N}\mathbf{1}_{T^Y_n<\infty})$.

To this aim let us define $(\mathcal{F}_n^Y)_{n\in\mathbb{Z}_+}$ the natural filtration associated to the stochastic process $(Y_n)_{n\in\mathbb{Z}_+}$, and
%
%
%
%
%
%
%
%

\begin{equation}
\begin{aligned}
U_n&=\mathbb{E}\Big(\Xi_n^A\Big|\mathcal{F}_n^Y\Big)\\
V_n&=\mathbb{E}\Big(\Xi^B_n\Big|\mathcal{F}_n^Y\Big).
\end{aligned}
\end{equation}
As previously, the quantity $U_{n}$ (resp. $V_{n}$) gives the probability, given $\mathcal{F}^Y_n$,  that a gene sampled among advantaged (resp. disadvantaged) individuals at time $n$, comes from any advantaged individual at time $0$. Note that  $$U_0=1\quad\quad \text{and}\quad\quad V_0=0.$$
These two stochastic processes are useful because they satisfy the following proposition :

\begin{proposition}\label{prop:rec-uv}
For $n<\inf(T_{0}^Y,T_N^Y)-1$, the sequence $ \begin{pmatrix}
U_{n} \\ 
V_{n}
\end{pmatrix}$ satisfies the following dynamics, assuming that $Y_n=k$ :

\begin{itemize}
\item If $Y_{n+1}=Y_n+1$ then $ \begin{pmatrix}
U_{n+1} \\ 
V_{n+1}
\end{pmatrix}=\bar{B}_{Y_n}^+  \begin{pmatrix}
U_{n} \\ 
V_{n}
\end{pmatrix}$ where \begin{equation}\label{eq-formB+}\bar{B}_k^+= \begin{pmatrix}
1-\frac{N-k}{2N(k+1)} & \frac{N-k}{2N(k+1)} \\ 
0 & 1
\end{pmatrix}.\end{equation}

\item  If $Y_{n+1}=Y_n-1$ then $ \begin{pmatrix}
U_{n+1} \\ 
V_{n+1}
\end{pmatrix}=\bar{B}_{Y_n}^-  \begin{pmatrix}
U_{n} \\ 
V_{n}
\end{pmatrix}$ where \begin{equation}\label{eq-formB-}\bar{B}_k^-= \begin{pmatrix}
1 & 0 \\ 
\frac{k}{2N(N-k+1)} & 1-\frac{k}{2N(N-k+1)}
\end{pmatrix}.\end{equation}

\item If $Y_{n+1}=Y_n$ then $ \begin{pmatrix}
U_{n+1} \\ 
V_{n+1}
\end{pmatrix}=\bar{B}_{Y_n}^{(0)}  \begin{pmatrix}
U_{n} \\ 
V_{n}
\end{pmatrix}$, where \begin{equation}\label{eq-formB0}\begin{aligned}\bar{B}_k^{(0)}&=\frac{k^2}{k^2+(1+s)(N-k)^2} \begin{pmatrix}
1-\frac{N-k}{2kN} & \frac{N-k}{2kN} \\ 
0 & 1
\end{pmatrix}\\&+\frac{(1+s)(N-k)^2}{k^2+(1+s)(N-k)^2} \begin{pmatrix}
1 & 0 \\ \frac{k}{2N(N-k)} & 1-\frac{k}{2N(N-k)}
\end{pmatrix}
\end{aligned}\end{equation}

\end{itemize}
\end{proposition}

As mentioned previously, this gives that if $T_{0}^Y<T_{N}^Y$ then $$ \begin{pmatrix}
U_{T_0^Y} \\ 
V_{T_0^Y}
\end{pmatrix}= \begin{pmatrix}
1& 0 \\ \frac{1}{2N^2}  & \frac{2N^2-1}{2N^2} 
\end{pmatrix} \begin{pmatrix}
U_{T_0^Y-1} \\ 
V_{T_0^Y-1}
\end{pmatrix}$$

and if $T_{N}^Y<T_{0}^Y$ then$ \begin{pmatrix}
U_{T_N^Y} \\ 
V_{T_N^Y}
\end{pmatrix}= \begin{pmatrix}
\frac{2N^2-1}{2N^2} & \frac{1}{2N^2} \\ 
0 & 1
\end{pmatrix} \begin{pmatrix}
U_{T_N^Y-1} \\ 
V_{T_N^Y-1}
\end{pmatrix}$.

\begin{proof} Recall that $U_n=\mathbb{P}(X_n^{(n)}\in\mathcal{Y}_0|X_0^{(n)}\sim \mathcal{U}(\mathcal{Y}_n),\mathcal{Y}_n,\mathcal{F}_n^Y)$ is the equal to the probability that a gene sampled uniformly in advantaged individuals at time $n$ comes from an advantaged individual at time $0$ (backwards in time). Similarly, $V_n$ is equal to the probability that a gene sampled uniformly in non-advantaged individuals at time $n$ comes from an advantaged individual at time $0$.

\medskip
Now if $Y_{n+1}=Y_n+1$ then as mentioned in the proof of Proposition \ref{prop:Y}, it means that at time $n$ an advantaged individual is chosen to be the mother, and a non advantaged individual is chosen to die. Therefore if $Y_n=k$ and $Y_{n+1}=Y_n+1=k+1$, then a gene sampled uniformly among advantaged individuals at time $n+1$ is either carried by an advantaged individual already present at time $n$ (with probability $k/(k+1)$), or is sampled in the newly advantaged individual (with probability $1/(k+1)$). In the first case a gene sampled in the individual comes from an advantaged individual at time $0$ with probability $U_n$. In the second case, it comes from an advantaged individual at time $0$ with probability $1/2\times U_n+1/2\times(\frac{k}{N}U_n+\frac{N-k}{N}V_n)=U_n(1-\frac{N-k}{2N})+V_n\frac{N-k}{2N}$. The first term of the left-hand side of the previous calculation corresponds to the case where the gene sampled at time $n+1$ comes from the (advantaged) mother of the considered individual (which occurs with probability $1/2$), and the second term corresponds to the case where the sampled gene comes from the father. Therefore \begin{align*}U_{n+1}&=\frac{k}{k+1}U_n+\frac{1}{k+1}\times\left(U_n\left(1-\frac{N-k}{2N}\right)+V_n\frac{N-k}{2N}\right)\\&=U_n\left(1-\frac{N-k}{2N(k+1)}\right)+V_n\frac{N-k}{2N(k+1)}.\end{align*} On the contrary, a gene sampled uniformly among non-advantaged individuals at time $n+1$ comes necessarily from an individual that was already non advantaged at time $n$, if $Y_{n+1}=Y_n+1$. Therefore if $Y_{n+1}=Y_n+1$, $V_{n+1}=V_n$. This ends the first point of the proposition.

\medskip
In the second case where $Y_{n+1}=Y_n-1$, the mother chosen at time $n$ must be a non-advantaged individual, while the dead individual must be an advantaged mutant. Therefore if $Y_{n+1}=Y_n-1$ and $Y_n=k$, then a gene sampled uniformly among advantaged individuals at time $n+1$ is necessarily carried by an advantaged individual already present at time $n$, which gives that $U_{n+1}=U_n$. Now a gene sampled uniformly at time $n+1$ among non advantaged individuals is either sampled among the previously non advantaged individuals (with probability $(N-k)/(N-k+1)$), or is sampled in the newly non advantaged individual (with probability $1/(N-k+1)$). In the first case this gene comes from an advantaged individual at time $0$ with probability $V_n$, and in the second case this gene comes from an advantaged individual with probability $1/2\times V_n+1/2\times(\frac{k}{N}U_n+\frac{N-k}{N}V_n)=U_n\frac{k}{2N}+V_n(1-\frac{k}{2N})$.
Therefore in the end when $Y_{n+1}=Y_n-1=k-1$, 
\begin{align*}V_{n+1}&=\frac{N-k}{N-k+1}V_n+\frac{1}{N-k+1}\times\left(U_n\frac{k}{2N}+V_n\left(1-\frac{k}{2N}\right)\right)\\&=U_n\frac{k}{2N(N-k+1)}+V_n\left(1-\frac{k}{2N(N-k+1)}\right)\end{align*}
which ends the second point of the proposition.

\medskip

In the last case where $Y_{n+1}=Y_n=k$, the mother and dying individual chosen at time $n$ are either both advantaged or either both non advantaged. Since the first event occurs with probability $\frac{k}{N}\times \frac{k}{k+(1+s)(N-k)}$ and the second event occurs with probability $\frac{(N-k)}{N}\times \frac{(1+s)(N-k)}{k+(1+s)(N-k)}$ this gives that the probability of the first event conditioned on the fact that $Y_{n+1}=Y_n=k$ is equal to $\frac{k^2}{k^2+(1+s)(N-k)^2}$ and the probability of the second event conditioned on the fact that $Y_{n+1}=Y_n=k$ is equal to $\frac{(1+s)(N-k)^2}{k^2+(1+s)(N-k)^2}$. Let us now assume that both the mother and the dying individual at time $n$ are advantaged. Then if one gene is sampled among non advantaged individuals at time $n+1$ then it is necessarily carried by a non advantaged individual already present at time $n$. So $V_{n+1}=V_n$. If a gene is sampled among advantaged individuals at time $n+1$ then it is either sampled in the new individual (with probability $1/k$) in which case it comes from an advantaged individual at time $0$ with probability $1/2\times U_n+1/2\times(\frac{k}{N}U_n+\frac{N-k}{N}V_n)=U_n\times(1-\frac{N-k}{2N})+\frac{N-k}{2N}V_n$ or it is sampled in an other advantaged individual  (with probability $(k-1)/k$) in which case it comes from an advantaged individual at time $0$ with probability $U_n$. In the end this gives that in this particular case where $Y_{n+1}=Y_n=k$ and both the mother and the dying individual at time $n$ are advantaged, \begin{align*}U_{n+1}&=\frac{k-1}{k}U_n+\frac{1}{k}\left(U_n\left(1-\frac{N-k}{2N}\right)+\frac{N-k}{2N}V_n\right)\\&=U_n\left(1-\frac{N-k}{2kN}\right)+V_n\frac{N-k}{2kN}.\end{align*}

Let us finally assume that both the mother and the dying individual at time $n$ are non advantaged individuals. Then if one gene is sampled among advantaged individuals at time $n+1$ then it necessarily is the copy of a gene already present in an advantaged individual of time $n$. So $U_{n+1}=U_n$. If a gene is sampled among non advantaged individuals at time $n+1$ then it is either sampled in the new individual (with probability $1/(N-k)$) in which case it comes from an advantaged individual at time $0$ with probability $1/2\times V_n+1/2\times(\frac{k}{N}U_n+\frac{N-k}{N}V_n)=U_n\frac{k}{2N}+V_n\times(1-\frac{k}{2N})$ or it is sampled in an other non advantaged individual  (with probability $(N-k-1)/(N-k)$) in which case it comes from an advantaged individual at time $0$ with probability $V_n$. In the end this gives that in this particular case where $Y_{n+1}=Y_n=k$ and both the mother and the dying individual at time $n$ are non advantaged individuals, \begin{align*}V_{n+1}&=\frac{N-k-1}{N-k}V_n+\frac{1}{N-k}\left(U_n\frac{k}{2N}+V_n\times\left(1-\frac{k}{2N}\right)\right)\\&=U_n\frac{k}{2N(N-k)}+V_n\left(1-\frac{k}{2N(N-k)}\right).\end{align*}
This ends the proof.
\end{proof}

\medskip


%

\begin{proposition}
\begin{description}
\item[$(i)$]  For each $k\in[\![1,N-1]\!]$, the matrices $\bar{B}_k^{(0)}$, $\bar{B}_k^+$ and $\bar{B}_k^-$ are stochastic, i.e. they admit the vector $\begin{pmatrix}1 \\ 1 \end{pmatrix}$ as eigenvector with eigenvalue $1$.

\item[$(ii)$] Being $2\times2$ stochastic matrices, for each $k\in[\![1,N-1]\!]$, the matrices $^t\bar{B}_k^{(0)}$, $^t\bar{B}_k^+$ and $^t\bar{B}_k^-$, admit the vector $\begin{pmatrix}-1 & 1 \end{pmatrix}$ as eigenvector. The respective associated eigenvalues are $l_0(k)=1-\frac{(2+s)k(N-k)}{2N(k^2+(1+s)(N-k)^2)}$, $l_+(k)=1-\frac{N-k}{2N(k+1)} $ and $l_-(k)=1-\frac{k}{2N(N-k+1)}$.
\end{description}

\end{proposition}

\begin{proof}
Both points of the proposition are immediate, recalling the formulas \eqref{eq-formB0}, \eqref{eq-formB+} and \eqref{eq-formB-} for the respective matrices $\bar{B}_k^{(0)}$, $\bar{B}_k^+$ and $\bar{B}_k^-$.

\end{proof}

\medskip
It is simpler to consider only the times at which the Markov chain $(Y_n)_{n\leq \inf(T_0^Y,T_N^Y)}$ jumps, and set

$$\begin{pmatrix}
\widetilde{U}_{n} \\ 
\widetilde{V}_{n}
\end{pmatrix}=\begin{pmatrix}
\mathbb{E}(U_{\tau_n}|\mathcal{F}^Z_n)\\ 
\mathbb{E}(V_{\tau_n}|\mathcal{F}^Z_n)
\end{pmatrix}.$$

Proposition \ref{prop:rec-uv} gives that

\begin{proposition}\label{prop:rec-uv-changed}

For $n<\inf(T_{0}^Y,T_N^Y)$, the sequence $\begin{pmatrix}
\widetilde{U}_{n} \\ 
\widetilde{V}_{n}
\end{pmatrix}$ satisfies the following dynamic. Assuming that $Z_n=k$,

\begin{itemize}
\item If $Z_{n+1}=Z_n+1$ then $\begin{pmatrix}
\widetilde{U}_{n+1} \\ 
\widetilde{V}_{n+1}
\end{pmatrix}=\widetilde{B}_{Z_n}^+ \begin{pmatrix}
\widetilde{U}_{n} \\ 
\widetilde{V}_{n}
\end{pmatrix}$ where $$\widetilde{B}_k^+= \begin{pmatrix}
1-\frac{N (s+2)-k (s+1)+1}{(k+1) (2 N+1) (s+2)}& \frac{N (s+2)-k (s+1)+1}{(k+1) (2 N+1) (s+2)} \\
 \frac{s+1}{(2N+1)(s+2)} & 1- \frac{s+1}{(2N+1)(s+2)}\\
\end{pmatrix}.$$

\item  If $Z_{n+1}=Z_n-1$ then $ \begin{pmatrix}
\widetilde{U}_{n+1} \\ 
\widetilde{V}_{n+1}
\end{pmatrix}=\widetilde{B}_{Z_n}^-  \begin{pmatrix}
\widetilde{U}_{n} \\ 
\widetilde{V}_{n}
\end{pmatrix}$ where $$\widetilde{B}_k^-= \begin{pmatrix}
1- \frac{1}{(2N+1)(s+2)} & \frac{1}{(2N+1)(s+2)} \\
 \frac{k+(N+1)(1+s)}{(2 N+1) (s+2) (N-k+1)} &1-  \frac{k+(N+1)(1+s)}{(2 N+1) (s+2) (N-k+1)} \\
\end{pmatrix}.$$

\item The eigenvalues associated to the left eigenvector $\begin{pmatrix}1&-1\end{pmatrix}$ of  $\widetilde{B}_k^+$ and $\widetilde{B}_k^+$ are respectively equal to $\lambda_+(k)=1-\frac{N+1}{(2N+1)(k+1)}$ and $\lambda_-(k)=1-\frac{N+1}{(2N+1)(N-k+1)}=\lambda_+(N-k)$.
\end{itemize}
\end{proposition}

\begin{proof} The time during which the Markov chain $(Y_n)_{n\in\mathbb{Z}_+}$ stays  in the state $k$ before jumping follows a geometric law (with value in $\mathbb{Z}_+$) with probability of success $$p_k\frac{k(N-k)}{N\left(\frac{1}{2+s}k+\frac{1+s}{2+s}(N-k)\right)}=\frac{(2+s)\times k(N-k)}{N(k+(1+s)\times(N-k))}$$ which was defined in Proposition \ref{prop:Y}.

Therefore if $Z_n=k$ and $Z_{n+1}=Z_n+1$ then \begin{align}\label{eq:Utilde1} \begin{pmatrix}
\widetilde{U}_{n+1} \\ 
\widetilde{V}_{n+1}
\end{pmatrix}&=p_k\bar{B}_k^+\sum_{l=0}^{\infty}\left((1-p_k)\bar{B}_k^{(0)}\right)^l  \begin{pmatrix}
\widetilde{U}_{n} \\ 
\widetilde{V}_{n}
\end{pmatrix}\\&=p_k\bar{B}_k^+\left(I-(1-p_k)\bar{B}_k^{(0)}\right)^{-1} \begin{pmatrix}
\widetilde{U}_{n} \\ 
\widetilde{V}_{n}
\end{pmatrix}.
\end{align} 

Hence if $Z_{n+1}=Z_n+1$ then $ \begin{pmatrix}
\widetilde{U}_{n+1} \\ 
\widetilde{V}_{n+1}
\end{pmatrix}=\widetilde{B}_{Z_n}^+ \begin{pmatrix}
\widetilde{U}_{n} \\ 
\widetilde{V}_{n}
\end{pmatrix}$ where $$\widetilde{B}_k^+=p_k\bar{B}_k^+\left(I-(1-p_k)\bar{B}_k^{(0)}\right)^{-1}.$$

As the matrix $\bar{B}_k^{(0)}$ is stochastic, so is the matrix $$p_k\left(I-(1-p_k)\bar{B}_k^{(0)}\right)^{-1}=p_k\sum_{l\geq0}(1-p_k)^l(\bar{B}_k^{(0)})^l.$$
Let us then write $$p_k\left(I-(1-p_k)\bar{B}_k^{(0)}\right)^{-1}=\begin{pmatrix}1-a&a\\b &1-b\end{pmatrix},$$ which implies that $$(1-p_k)\bar{B}_k^{(0)}=\begin{pmatrix}1+p_k\frac{1-b}{a+b-1}&-p_k\frac{a}{a+b-1}\\-p_k\frac{b}{a+b-1}&1+p_k\frac{1-a}{a+b-1}\end{pmatrix}$$

This equation, together with Equation \eqref{eq-formB0} gives first that $$\frac{a}{b}=\frac{1}{1+s}$$ and second that 
$$(1-p_k)\frac{k(N-k)}{2N(k^2+(1+s)(N-k)^2)}=-p_k\frac{a}{a(2+s)-1}.$$

Therefore \begin{equation}\label{eq-a}a=\frac{(1-p_k)k(N-k)}{(2+s)(1-p_k)k(N-k)+p_k2N(k^2+(1+s)(N-k)^2)}=\frac{1}{(2N+1)(s+2)  }\end{equation} and $$b=(1+s)a=\frac{1+s}{(2N+1)(s+2)  }$$ where the second equality in Equation \eqref{eq-a} is obtained using Mathematica (see Section \ref{sec:Mathematica}).

Now recall that 
\begin{equation}\bar{B}_k^+= \begin{pmatrix}
1-\frac{N-k}{2N(k+1)} & \frac{N-k}{2N(k+1)} \\ 
0 & 1
\end{pmatrix}\end{equation}
therefore 
$$\widetilde{B}_k^+= \begin{pmatrix}
1-\frac{N (s+2)-k (s+1)+1}{(k+1) (2 N+1) (s+2)}& \frac{N (s+2)-k (s+1)+1}{(k+1) (2 N+1) (s+2)} \\
 \frac{s+1}{(2N+1)(s+2)} & 1- \frac{s+1}{(2N+1)(s+2)}\\
\end{pmatrix}.$$

which gives the first point of the proposition.

Similarly if $Z_n=k$ and $Z_{n+1}=Z_n-1$ then \begin{align}\label{eq:Utilde2} \begin{pmatrix}
\widetilde{U}_{n+1} \\ 
\widetilde{V}_{n+1}
\end{pmatrix}&=p_k\bar{B}_k^-\sum_{l=0}^{\infty}\left((1-p_k)\bar{B}_k^{(0)}\right)^l  \begin{pmatrix}
\widetilde{U}_{n} \\ 
\widetilde{V}_{n}
\end{pmatrix}\\&=p_k\bar{B}_k^-\left(I-(1-p_k)\bar{B}_k^{(0)}\right)^{-1} \begin{pmatrix}
\widetilde{U}_{n} \\ 
\widetilde{V}_{n}
\end{pmatrix}
\end{align} 
where $I$ is the identity matrix of size $2$. Therefore, recalling that $$\bar{B}_k^-= \begin{pmatrix}
1 & 0 \\ 
\frac{k}{2N(N-k+1)} & 1-\frac{k}{2N(N-k+1)}
\end{pmatrix}$$ we get that if $Z_{n+1}=Z_n-1$ then $ \begin{pmatrix}
\widetilde{U}_{n+1} \\ 
\widetilde{V}_{n+1}
\end{pmatrix}=\widetilde{B}_{Z_n}^-  \begin{pmatrix}
\widetilde{U}_{n} \\ 
\widetilde{V}_{n}
\end{pmatrix}$ where $$\widetilde{B}_k^-=p_k\bar{B}_k^-\left(I-(1-p_k)\bar{B}_k^{(0)}\right)^{-1}= \begin{pmatrix}
1- \frac{1}{(2N+1)(s+2)} & \frac{1}{(2N+1)(s+2)} \\
 \frac{k+(N+1)(1+s)}{(2 N+1) (s+2) (N-k+1)} &1-  \frac{k+(N+1)(1+s)}{(2 N+1) (s+2) (N-k+1)} \\
\end{pmatrix}.$$
which gives the second point of the proposition.
\end{proof}

\medskip

The previous proposition gives in particular that

\begin{proposition} 
\begin{description}
\item[$(i)$] For any time $n\leq \inf(T_0^Z,T_N^Z)$,
\begin{equation} \label{eq-D}
D_n=\widetilde{U}_n-\widetilde{V}_n=\prod_{l=0}^{n-1}\lambda_{\epsilon_l}(Z_l)=\prod_{k=1}^{N-1}\lambda_+(k)^{S^+_n(k)}\lambda_-(k)^{S^-_n(k)}.\end{equation} where $S^+_n(k)$, and $S^-_n(k)$ are respectively the number of transition from $k$ to $k+1$, and from $k$ to $k-1$ in the trajectory of $(Z_l)_{l\leq n}$, and $\epsilon_l={-}$ if $Z_{l+1}=Z_l-1$ and $\epsilon_l={+}$ if $Z_{l+1}=Z_l+1$.
\item[$(ii)$] One has 
\begin{equation}\label{eq-UV}
\begin{aligned}
\widetilde{U}_n&=1-\sum_{l=0}^{n-1}\alpha_{\epsilon_l}(Z_l)D_l\\
\widetilde{V}_n&=\sum_{l=0}^{n-1}\beta_{\epsilon_l}(Z_l)D_l
\end{aligned}
\end{equation}
where \begin{equation*}
\begin{aligned}\alpha_+(k)&=\frac{-k (s+1)+N (s+2)+1}{(k+1) (2 N+1) (s+2)}=\frac{1}{(2N+1)(s+2)}+\frac{N-k}{(k+1)(2N+1)},\\ \alpha_-(k)&=\frac{1}{(2N+1)(s+2)}, \\\beta_+(k)&=\frac{s+1}{(2N+1)(s+2)}, \quad\quad \text{ and } \\\beta_-(k)&=-\frac{k+N s+N+s+1}{(2 N+1) (s+2) (k-N-1)}=\frac{s+1}{(2N+1)(s+2)}+\frac{k}{(2N+1)(N-k+1)}.
\end{aligned}
\end{equation*}
\end{description}
\end{proposition}

\begin{proof}

From Proposition \ref{prop:rec-uv-changed} one has the following system of equations :
\begin{equation}\label{eq-matrix-UV} \left\{\begin{matrix}
\widetilde{U}_{n+1}=(1-\alpha_{\epsilon_n}(Z_n))\widetilde{U}_n+\alpha_{\epsilon_n}(Z_n)\widetilde{V}_n\\
\widetilde{V}_{n+1}=\beta_{\epsilon_n}(Z_n) \widetilde{U}_n+(1-\beta_{\epsilon_n}(Z_n))\widetilde{V}_n
\end{matrix}\right..\end{equation}

Therefore first $$\widetilde{U}_{n+1}-\widetilde{V}_{n+1}=(1-\alpha_{\epsilon_n}(Z_n)-\beta_{\epsilon_n}(Z_n))(\widetilde{U}_{n}-\widetilde{V}_{n})=\lambda_{\epsilon_n}(Z_n)(\widetilde{U}_{n}-\widetilde{V}_{n}), $$ which gives Equation \eqref{eq-D}.

Now using the first equation of \eqref{eq-matrix-UV}, one has :
$$\widetilde{U}_{n+1}-\widetilde{U}_n=-\alpha_{\epsilon_n}(Z_n)D_n,$$

which gives that
$$\widetilde{U}_n=1-\sum_{l=0}^{n-1}\alpha_{\epsilon_l}(Z_l)D_l \quad\text{as $\widetilde{U}_0=1$.}$$

Similarly using the second equation of  \eqref{eq-matrix-UV} gives that 
$$\widetilde{V}_{n+1}-\widetilde{V}_n=\beta_{\epsilon_n}(Z_n)D_n$$ and therefore 
$$\widetilde{V}_n=\sum_{l=0}^{n-1}\beta_{\epsilon_l}(Z_l)D_l\quad\text{as $\widetilde{V}_0=0$.}$$
\end{proof}

\medskip


In this article we focus on the situation in which the mutation is already present in a significant proportion of the individuals, and study the order of magnitude of the weight of ancestors once the advantageous mutation has spread in the population. This leads us to use a continuous approximation, in which the previous equations will be replaced by differential equations.
The first step of our study consists in studying the expectation of the difference of weights between advantaged and disadvantaged individuals. 
Recall that $T^Z_y=\inf\{n:Z_n=y\}$. 

\begin{proposition}\label{prop-X}
For any $0<a<b<1$,

$$\mathbb{E}(D_{T^Z_{\lfloor bN\rfloor}}\mathbf{1}_{T^Z_{\lfloor bN\rfloor}<\infty}| Z_0=\lfloor aN\rfloor)\rightarrow_{N\rightarrow\infty}\frac{\left(\frac{a}{b}\right)^{\frac{1+s}{2s}}}{\left(\frac{1-a}{1-b}\right)^{\frac{1}{2s}}}$$
when $N$ goes to infinity.
\end{proposition}

Although we are particularly interested in the case where $b=1$, we start by considering the case where $b>1$, in order to simplify the approximation. The limit where $b$ goes to $1$ will be tackled in the final step of the proof of  Theorem \ref{thm}.



\begin{proof}
Define for any $k\in[\![1,\lfloor bN\rfloor]\!]$, $$\phi_D(k)=\mathbb{E}(D_{T^Z_{\lfloor bN\rfloor}}\mathbf{1}_{T^Z_{\lfloor bN\rfloor}<\infty}| Z_0=k) \quad\text{ and }\quad
\psi_D(k)=\frac{\left(\frac{k}{N}\right)^{\frac{1+s}{2s}}}{\left(1-\frac{k}{N}\right)^{\frac{1}{2s}}}\times \frac{(1-b)^{\frac{1}{2s}}}{b^{\frac{1+s}{2s}}}.$$ Our aim is to prove that the functions $\phi_D$ and $\psi_D$ are close to each other on $[\lfloor aN\rfloor,\lfloor bN\rfloor]$ for large $N$.

Let us define the infinitesimal generator $\mathcal{L}$ on the set of real valued functions $g$ on $[\![0,N]\!]$ vanishing in $0$ and $N $ by $$\mathcal{L}g(k)=\frac{1+s}{2+s}\lambda_+(k)g(k+1)+\frac{1}{2+s}\lambda_-(k)g(k-1)-g(k)$$ for any $k\in[\![1,N-1]\!]$. Recall that $\lambda_+(k)=1-\frac{N+1}{(k+1)(2N+1)}$ and $\lambda_-(k)=1-\frac{N+1}{(N-k+1)(2N+1)}$. Therefore  \begin{align*}\mathcal{L}g(k)&=\frac{1+s}{2+s}\left(g(k+1)-g(k)\right)+\frac{1}{2+s}\left(g(k-1)-g(k)\right)\\&-\frac{1+s}{2+s}\frac{N+1}{(k+1)(2N+1)}g(k+1)-\frac{1}{2+s}\frac{N+1}{(N-k+1)(2N+1)}g(k-1).\end{align*}

By definition, the function $\phi_D$ is such that for any $k\in[\![1,\lfloor bN\rfloor-1]\!]$,
$$\mathcal{L}\phi_D(k)=0,$$ 
and $\phi_D(\lfloor bN\rfloor)=1$. 

Now the function $\psi_D$ is such that $\psi_D(\lfloor bN\rfloor)=1$. What is more,  for any $1\leq k \leq \lfloor bN\rfloor$, $$\psi_D(k)=f_D\left(\frac{k}{N}\right)$$ where $f_D(x)=C_b\frac{x^{\frac{1+s}{2s}}}{(1-x)^{\frac{1}{2s}}}$, where $C_b=\frac{(1-b)^{\frac{1}{2s}}}{(b^{\frac{1+s}{2s}}}$ for all $x\in[0,b]$. One has for all $x\in[0,b]$,
 \begin{align*}f'_D(x)&=f_D(x)\left[\frac{1+s}{2s}\frac{1}{x}+\frac{1}{2s}\frac{1}{1-x}\right]\quad\text{and}\\f''_D(x)&=f_D(x)\left[\left(\frac{1+s}{2s}\frac{1}{x}+\frac{1}{2s}\frac{1}{1-x}\right)^2-\frac{1+s}{2s}\frac{1}{x^2}+\frac{1}{2s}\frac{1}{(1-x)^2}\right].\end{align*}
 
 Hence $$\psi_D(k\pm1)=\psi_D(k)\pm\psi_D(k)\left[\frac{1+s}{2s}\frac{1}{k}+\frac{1}{2s}\frac{1}{N-k}\right]+R\left(\frac{k}{N}\right)$$ where the function $R$ is such that there exists a positive constant $C$ such that $|R(x)|\leq \frac{C}{N^2}$ for all $x\in[a,b]$.
 
 Therefore  for any $1\leq k \leq \lfloor bN\rfloor-1$  \begin{align*}\mathcal{L}\psi_D(k)&=\frac{1}{2+s}\frac{1}{2N}\psi_D(k)\left[\frac{1+s}{\frac{k}{N}}+\frac{1}{1-\frac{k}{N}}\right]-\frac{1}{2+s}\psi_D(k)\left(\frac{1+s}{2(k+1)}+\frac{1}{2(N-k+1)}\right)+\widetilde{R}\left(\frac{k}{N}\right)\\&=\frac{1}{2+s}\psi_D(k)\left[\frac{1+s}{2}\left(\frac{1}{k}-\frac{1}{k+1}\right)+\frac{1}{2}\left(\frac{1}{N-k}-\frac{1}{N-k+1}\right)\right]+\widetilde{R}\left(\frac{k}{N}\right)\end{align*}

where the function $\widetilde{R}$  is such that there exists a positive constant $\widetilde{C}$ such that $|\widetilde{R}(x)|\leq \frac{C}{N^2}$ for all $x\in[a,b]$.
Hence there exists a positive constant $C'$ such that for all $k\in[\lfloor aN\rfloor,\lfloor bN\rfloor]$,
$$|\mathcal{L}\psi(k)|\leq \frac{C}{N^2}.$$

Let us  now define the function $\bar{\psi}$ by $\bar{\psi}_D(k)=\psi_D(k)$ for all $k\in[\lfloor aN/2\rfloor,\lfloor bN\rfloor]$ and $\bar{\psi}_D(k)=\phi_D(k)$ for all $k\in[0,\lfloor aN/2\rfloor[$. Our aim from now is to prove that the functions $\phi_D$ and $\bar{\psi}_D$ are close to each other on $[\lfloor aN\rfloor, \lfloor bN\rfloor]$.

Note that since $\bar{\psi_D}-\phi_D$ vanishes in $0$ and $N$,  $$-(\bar{\psi}_D-\phi_D)=((-\mathcal{L})^{-1})\mathcal{L})(\bar{\psi}_D-\phi_D),$$

where $(-\mathcal{L})^{-1}=\sum_{n\geq 0} (I+\mathcal{L})^n=\sum_{n\geq 0} Q^n$ where $Q=(Q_{ij})_{1\leq i,j\leq \lfloor bN\rfloor-1}$ is defined by \begin{align*}Q_{k,k+1}&=\frac{1+s}{2+s}\lambda_+(k), \quad\quad\text{for all $k\in[1,\lfloor bN-2\rfloor]$}\\Q_{k,k-1}&=\frac{1}{2+s}\lambda_-(k) \quad\quad\text{for all $k\in[2,\lfloor bN\rfloor-1]$}\\
Q_{k,l}&=0 \quad \text{elsewhere.}\end{align*}

We know that \begin{equation}\label{eq-Ldiff}\left\{\begin{aligned}&\mathcal{L}(\bar{\psi}_D-\phi_D)(k)=O\left(\frac{1}{N^2}\right) \quad\text{ uniformly for $k\in[\![1,N]\!]\setminus\{\lfloor a N/2\rfloor-1,\lfloor a N/2\rfloor,\lfloor a N/2\rfloor +1\}$, and}\\&\mathcal{L}(\bar{\psi}_D-\phi_D)(k)\leq 1 \text{  for $k\in\{\lfloor a N/2\rfloor-1,\lfloor a N/2\rfloor,\lfloor aN/2\rfloor +1\}$.}\end{aligned}\right.\end{equation}

Note now that since $\lambda_+(k)$ and $\lambda_-(k)$ are in $[0,1]$,  $$\big(\sum_{n\geq 0} Q^n\big)(x,y)\leq\mathbb{E}\big(\sum_{n\geq0} \mathbf{1}_{Z_n=y}|Z_0=x\big),$$ and from Markov property,

$$\mathbb{E}\big(\sum_{n\geq0} \mathbf{1}_{Z_n=y}|Z_0=x\big)=\mathbb{P}(T^Z_y<\infty|Z_0=x)\mathbb{E}(\sum_{n\geq0}  \mathbf{1}_{Z_n=y}|Z_0=y)<C(s) \mathbb{P}_x(T^Z_y<\infty)$$ as $s>0$. 

Besides, recall from Proposition \ref{prop-Z} that $\mathbb{P}(T^Z_{ a N/2}<\infty|Z_0= a N)$ decreases exponentially with $N$, when $ a$ is fixed.

Therefore \begin{align*}
\left|\bar{\psi}_D-\phi_D\right|(k)&=\big|\big(\sum_{n\geq 0} Q^n\big)\big(\mathcal{L}(\bar{\psi}_D-\phi_D)\big)\big|(k)\\&\leq\big(\sum_{n\geq 0} Q^n\big)\big|\mathcal{L}(\bar{\psi}_D-\phi_D)\big|(k)\\&\leq\sum_{j\in[\![0,N]\!]}\sum_{n\geq 0} Q^n_{kj}\big|\big(\mathcal{L}(\bar{\psi}_D-\phi_D)\big)(j)\big|\\&=\sum_{j\in\{\lfloor  a N/2\rfloor-1,\lfloor a N/2\rfloor,\lfloor a N/2\rfloor +1\}}\sum_{n\geq 0} Q^n_{kj}\big|\big(\mathcal{L}(\bar{\psi}_D-\phi_D)\big)(j)\big|\\&+\quad\sum_{j\notin\{\lfloor  a N/2\rfloor-1,\lfloor a N/2\rfloor,\lfloor a N/2\rfloor +1\}}\sum_{n\geq 0} Q^n_{kj}\big|\big(\mathcal{L}(\bar{\psi}_D-\phi_D)\big)(j)\big|.\end{align*}

The first quantity has an exponential bound from Proposition \ref{prop-Z}. The second quantity is bounded by $C/N$, from Equation \ref{eq-Ldiff} and since $\sum_j \sum_n Q^n_{kj}= \mathbb{E}(T_{0,N}^Z|Z_0=k)\leq C(s) N$. Therefore there exists a constant $C$ such that $\left|\bar{\psi}_D-\phi_D\right|(k)\leq C/N$ for all $k \in[ a N, bN]$.
\end{proof}

Pushing the same approach further and using the previous proposition, we get that

\begin{proposition}\label{prop-main}For any $0<a<b<1$,
\begin{align*}\mathbb{E}(\widetilde{U}_{T_{0,b N}}\mathbf{1}_{T^Z_{\lfloor bN\rfloor}<\infty}|Z_0=\lfloor aN\rfloor)&\rightarrow_{N\rightarrow\infty} \frac{\left(\frac{a}{b}\right)^{\frac{1+s}{2s}}}{\left(\frac{1-a}{1-b}\right)^{\frac{1}{2s}}}+\left(\int_{a}^{b}\frac{(1-u)^{\frac{1}{2s}}}{u^{\frac{1+s}{2s}}}\left[\frac{1}{2}+\frac{1}{2s}\frac{1}{1-u}\right]du\right)\times \frac{a^{\frac{1+s}{2s}}}{(1-a)^{\frac{1}{2s}}}\end{align*}

\begin{align*}\mathbb{E}(\widetilde{V}_{T_{0,b N}}\mathbf{1}_{T^Z_{\lfloor bN\rfloor}<\infty}|Z_0=\lfloor aN\rfloor)\rightarrow_{N\rightarrow\infty}\left(\int_{a}^{b}\frac{(1-u)^{\frac{1}{2s}}}{u^{\frac{1+s}{2s}}}\left[\frac{1}{2}+\frac{1}{2s}\frac{1}{1-u}\right]du\right)\times \frac{a^{\frac{1+s}{2s}}}{(1-a)^{\frac{1}{2s}}}\end{align*}
\end{proposition}

\begin{proof} For any $k\in[a N, b N]$ let us denote $\phi_{\widetilde{V}}(k)=\mathbb{E}(\widetilde{V}_{T_{0,b N}}|Z_0=k)$.  The function $\phi_{\widetilde{V}}$ now satisfies for any $k\in]\!]1, bN[\![$,
$$\mathcal{L}\phi_{\widetilde{V}}(k)=-\frac{1+s}{2+s}\beta_+(k)-\frac{1}{2+s}\beta_-(k).$$
Moreover, when $N$ is large, $\beta_+(k)\sim\frac{1+s}{2N(2+s)}$ and $\beta_-(k)\sim\frac{1+s}{2N(2+s)}+\frac{k}{2N(N-k)}$. Therefore

\begin{align*}\mathcal{L}\phi_{V}(k)=-\frac{1+s}{2N(2+s)}-\frac{1}{2+s}\frac{k}{2N(N-k)}+O(1/N^2)=-\frac{1}{2N(2+s)}\frac{N+s(N-k)}{N-k}+O(1/N^2)\end{align*}
for all $k\in[\![1,bN]\!]$.

Now let $$\psi_{\widetilde{V}}(k)=\left(\int_{k/N}^{b}\frac{(1-u)^{\frac{1}{2s}}}{u^{\frac{1+s}{2s}}}\left[\frac{1}{2}+\frac{1}{2s}\frac{1}{1-u}\right]du\right)\times \frac{\left(\frac{k}{N}\right)^{\frac{1+s}{2s}}}{(1-\frac{k}{N})^{\frac{1}{2s}}}.$$

Our aim from here is to prove that the functions $\phi_{V}$ and $\psi_{V}$ are close to each other.

Setting $\psi_{\widetilde{V}}(k)=f_{\widetilde{V}}(k/N)$, we have

$$f_{V}(t)=\left(\int_{t}^{b}\frac{(1-u)^{\frac{1}{2s}}}{u^{\frac{1+s}{2s}}}\left[\frac{1}{2}+\frac{1}{2s}\frac{1}{1-u}\right]du\right)\times \frac{t^{\frac{1+s}{2s}}}{(1-t)^{\frac{1}{2s}}}$$

hence

$$f_{\widetilde{V}}'(t)=f_{\widetilde{V}}(t)\left[\frac{1+s}{2s}\frac{1}{t}+\frac{1}{2s}\frac{1}{2(1-t)}\right]-\frac{1}{2}-\frac{1}{2s}\frac{1}{1-t}$$
and 
$$f_{\widetilde{V}}''(t)=f_{\widetilde{V}}(t)\left[\frac{1+s}{2s}\frac{1}{t}+\frac{1}{2s}\frac{1}{1-t}\right]-\frac{1+s}{2(2+s)}-\frac{1}{2+s}\frac{t}{2(1-t)}.$$

Therefore $$\psi_{\widetilde{V}}(k\pm1)=\psi_{\widetilde{V}}(k)\pm\frac{1}{N}\left[\psi_{\widetilde{V}}(k)\left[\frac{1+s}{2+s}\frac{1}{2\frac{k}{N}}+\frac{1}{2+s}\frac{1}{2(1-\frac{k}{N})}\right]-\frac{1+s}{2(2+s)}-\frac{1}{2+s}\frac{\frac{k}{N}}{2(1-\frac{k}{N})}\right]+R(k/N)$$ where the function $R$ is such that there exists a positive constant $C$ such that $|R(x)|\leq \frac{C}{N^2}$. Hence, since $\psi_{\widetilde{V}}$ is bounded, 

\begin{align*}\mathcal{L}\psi_{\widetilde{V}}(k)=-\frac{1}{2+s}\left[\frac{s}{2N}+\frac{1}{2(N-k)}\right]+O\left(\frac{1}{N^2}\right)=-\frac{1}{2+s}\frac{s(N-k)+N}{2N(N-k)}+O\left(\frac{1}{N^2}\right).\end{align*} for all $k\in[\![1,bN]\!]$.

Therefore $$\mathcal{L}(\psi_{\widetilde{V}}-\phi_{\widetilde{V}})(k)=O\left(\frac{1}{N^2}\right)$$ for all $k\in[\![1,bN]\!]$.

Now as in the proof of Proposition \ref{prop-X}, one can introduce a function $\bar{\psi}_{\tilde{V}}$ which coincides with $\psi_{\tilde{V}}$ above $\lfloor aN/2\rfloor$ and with $\phi_{\tilde{V}}$ below. Next, reasoning exactly as in the proof of Proposition \ref{prop-X}, the fact that $\mathbb{P}(T^Z_{aN.2}|Z_0=aN)$ decreases exponentially with $N$ when $a$ is fixed gives that there exists a constant $C$ such that $\left|\bar{\psi}_V-\phi_V\right|(k)\leq C/N$ for all $k \in[\![aN,bN]\!]$ which gives the result for $\tilde{V}$. The expression for $\tilde{U}$ then follows by Proposition \ref{prop-X}.

\end{proof}
We can now complete the proof of our theorem :
\begin{theorem}\label{thm} 
\begin{align*}\mathbb{E}(\Xi^A_{T^Y_N}\mathbf{1}_{T^Y_N<\infty}|Y_0=\lfloor aN\rfloor)&\rightarrow_{N\rightarrow\infty} \frac{a^{\frac{1+s}{2s}}}{(1-a)^{\frac{1}{2s}}}\left(\int_{a}^{1}\frac{(1-u)^{\frac{1}{2s}}}{u^{\frac{1+s}{2s}}}\left[\frac{1}{2}+\frac{1}{2s}\frac{1}{1-u}\right]du\right). \end{align*}
\end{theorem}

\begin{proof}
First, by definition of $\tilde U$, $\mathbb{E}(\Xi^A_{T^Y_N}\mathbf{1}_{T^Y_N<\infty}|Y_0=\lfloor aN\rfloor)=\mathbb{E}(\tilde U_{T^Z_N}\mathbf{1}_{T^Z_N<\infty}|Y_0=\lfloor aN\rfloor)$.\\
Note that from Equation \eqref{eq-UV}, the sequences $(\widetilde{U}_n)_{n\in\mathbb{Z}_+}$ and $(\widetilde{V}_n)_{n\in\mathbb{Z}_+}$ are respectively  decreasing and increasing. In particular for any $\epsilon>0$, if $T^Z_N<\infty$,
$$\widetilde{V}_{T^Z_{\lfloor N(1-\epsilon)\rfloor}}\leq \widetilde{V}_{T^Z_{N-1}}\leq\tilde{U}_{T^Z_{N-1}}=\tilde{U}_{T^Z_N}\leq \tilde{U}_{T^Z_N}.$$ 
Therefore 
$$\mathbb{E}\left(\widetilde{V}_{T^Z_{\lfloor N(1-\epsilon)\rfloor}}\mathbf{1}_{T^Z_N<\infty}\right)\leq \mathbb{E}\left(\tilde{U}_{T^Z_{N-1}}\mathbf{1}_{T^Z_N<\infty}\right)=\mathbb{E}\left(\tilde{U}_{T^Z_N}\mathbf{1}_{T^Z_N<\infty}\right)\leq \mathbb{E}\left(\tilde{U}_{T^Z_{\lfloor N(1-\epsilon)\rfloor}}\mathbf{1}_{T^Z_N<\infty}\right).$$ 
What is more, $$\mathbb{E}\left(\widetilde{V}_{T^Z_{\lfloor N(1-\epsilon)\rfloor}}\mathbf{1}_{T^Z_N<\infty}\right)\geq \mathbb{E}\left(\widetilde{V}_{T^Z_{\lfloor N(1-\epsilon)\rfloor}}\mathbf{1}_{T^Z_{\lfloor N(1-\epsilon)\rfloor}<\infty}\right)-\mathbb{P}(T^Z_N=\infty,T^Z_{\lfloor N(1-\epsilon)\rfloor}<\infty),$$ while $$\mathbb{E}\left(\tilde{U}_{T^Z_{\lfloor N(1-\epsilon)\rfloor}}\mathbf{1}_{T^Z_N<\infty}\right)\leq \mathbb{E}\left(\tilde{U}_{T^Z_{\lfloor N(1-\epsilon)\rfloor}}\mathbf{1}_{T^Z_{\lfloor N(1-\epsilon)\rfloor}<\infty}\right).$$

Therefore since $\mathbb{P}(T^Z_N=\infty,T^Z_{\lfloor N(1-\epsilon)\rfloor}<\infty)$ goes to $0$ when  $N$ goes to infinity, from Proposition \ref{prop-main},
$$\limsup_{N\rightarrow\infty} \mathbb{E}\left(\tilde{U}_{T^Z_{N}}\mathbf{1}_{T^Z_N<\infty}\right)\leq \frac{\left(\frac{a}{(1-\epsilon)}\right)^{\frac{1+s}{2s}}}{\left(\frac{1-a}{\epsilon}\right)^{\frac{1}{2s}}}+\left(\int_{a}^{1-\epsilon}\frac{(1-u)^{\frac{1}{2s}}}{u^{\frac{1+s}{2s}}}\left[\frac{1}{2}+\frac{1}{2s}\frac{1}{1-u}\right]du\right)\times \frac{a^{\frac{1+s}{2s}}}{(1-a)^{\frac{1}{2s}}}$$

and $$\liminf_{N\rightarrow\infty} \mathbb{E}\left(\tilde{U}_{T^Z_{N}}\mathbf{1}_{T^Z_N<\infty}\right)\geq \left(\int_{a}^{1-\epsilon}\frac{(1-u)^{\frac{1}{2s}}}{u^{\frac{1+s}{2s}}}\left[\frac{1}{2}+\frac{1}{2s}\frac{1}{1-u}\right]du\right)\times \frac{a^{\frac{1+s}{2s}}}{(1-a)^{\frac{1}{2s}}}.$$

Letting $\epsilon$ go to $0$ gives the result.
\end{proof}

\appendix

\section{Mathematica calculations}\label{sec:Mathematica}

These calculations made by Mathematica are the justification for the calculation of $a$ given in Equation \eqref{eq-a}.

\begin{doublespace}
\noindent\(\pmb{\text{pk}=(2+s)*k*(N-k)/(N*(k+(1+s)*(N-k)));}\)
\end{doublespace}

\begin{doublespace}
\noindent\(\pmb{\text{ak}=((1-\text{pk})*k*(N-k))/((2+s)*(1-\text{pk})*k*(N-k)+\text{pk}*2*N*(k*k+(1+s)*(N-k)}\\
\pmb{*(N-k)))}\)
\end{doublespace}

\begin{doublespace}
\noindent\(\frac{k (-k+N) \left(1-\frac{k (-k+N) (2+s)}{N (k+(-k+N) (1+s))}\right)}{\frac{2 k (-k+N) (2+s) \left(k^2+(-k+N)^2 (1+s)\right)}{k+(-k+N)
(1+s)}+k (-k+N) (2+s) \left(1-\frac{k (-k+N) (2+s)}{N (k+(-k+N) (1+s))}\right)}\)
\end{doublespace}

\begin{doublespace}
\noindent\(\pmb{\text{ak}=\text{Factor}[\text{FullSimplify}[\text{ak}]]}\)
\end{doublespace}

\begin{doublespace}
\noindent\(\frac{1}{(1+2 N) (2+s)}\)
\end{doublespace}

\bibliographystyle{abbrv}
\bibliography{biblio}

\end{document}